\documentclass[12pt]{amsart}
\usepackage{amsmath,amssymb}
\usepackage{eufrak}
\usepackage{amsthm}
\usepackage{enumitem}
\usepackage{graphicx}
\usepackage{calc}
\usepackage{graphics}
\usepackage{color}
\usepackage{xcolor}
\usepackage{hyperref}
\usepackage{youngtab}
\usepackage{marginnote}
\usepackage{lscape}
\usepackage{tikz}
\usetikzlibrary{calc,shadings,patterns}
\usetikzlibrary{matrix}
\usepackage{faktor}
\usepackage{hyperref}
\usepackage[normalem]{ulem}
\usepackage{bm}
\usepackage{afterpage}
\usepackage{mathrsfs}
\usepackage[all]{xy}
\usetikzlibrary{arrows}
\usetikzlibrary{matrix}
\usepackage{verbatim}
\usepackage{setspace}

\usepackage{multirow}
\usetikzlibrary{decorations.pathreplacing}
\usepackage{enumitem,pifont,xcolor}
\definecolor{myblue}{RGB}{0,29,119}

\newtheorem{theorem}{Theorem}[section]
\newtheorem{proposition}[theorem]{Proposition}

\newtheorem{lemma}[theorem]{Lemma}
\theoremstyle{definition}
\newtheorem{definition}[theorem]{Definition}
\newtheorem{example}[theorem]{Example}

\newtheorem{remark}[theorem]{Remark}

\newtheorem*{theorem*}{Theorem}
\newtheorem{question}[theorem]{Question}

\makeatletter
\@addtoreset{equation}{section}
\makeatother

\newcommand{\Pp}{{\mathbb P}} 

\newcommand{\cI}{{\mathcal I}}

\newcommand{\cP}{{\mathcal P}}

\DeclareMathOperator{\pdim}{p\,\! dim}
\DeclareMathOperator{\indim}{in\,\! dim}
\DeclareMathOperator{\Hom}{Hom}
\DeclareMathOperator{\ext}{Ext}

\providecommand{\AMS}{$\mathcal{A}$\kern-.1667em%
\lower.25em\hbox{$\mathcal{M}$}\kern-.125em$\mathcal{S}$}

\begin{document}


\title[Exceptional Sequences and Idempotent Functions]{Exceptional Sequences and Idempotent Functions}
\begin{abstract}
We prove that there is a one to one correspondence between the following three sets: idempotent functions on a set of size $n$, complete exceptional sequences of linear radical square zero Nakayama algebras of rank $n$ and rooted labeled forests with $n$ nodes and height of at most one. Therefore, the number of exceptional sequences is given by the sum $\sum\limits^n_{j=1}\binom{n}{j}j^{n-j}$.
\end{abstract}

\author{Emre SEN}

\let\thefootnote\relax\footnotetext{
Contact: \href{mailto:emre-sen@uiowa.edu}{emre-sen@uiowa.edu} }

\maketitle

\section{Introduction}
An indecomposable module $M$ of mod-$\!A$ is called exceptional if $\Hom_A(M,M)\cong k$, $\ext^i_A(M,M)=0$ for $i\geq 1$ where $A$ is a finite dimensional algebra over algebraically closed field $k$.

The notion of exceptional sequences was introduced by  Gorodentsev and Rudakov \cite{goro} in the study of exceptional vector bundles over $\Pp^2$. This was carried into the quiver representation set up by Crawley-Boevey \cite{craw} and Ringel \cite{ringel2}. Studying exceptional sequences of either coherent sheaves and their derived categories or quiver representations and their combinatorics are important. There are many works in those directions such as  \cite{rudakov}, \cite{melt},\cite{len}, \cite{ralf},\cite{hille} to name a few.

Enumerative problems are also of interest. For instance, the number of complete exceptional sequences of hereditary algebras of types $\mathbb{A}_n$ and $\mathbb{D}_n$ are given by $(n+1)^{n-1}$ and $2(n-1)^n$ \cite{seidel}. Combinatorial interpretations of those can be found in the papers \cite{araya},\cite{ringel} and \cite{igusa} for instance.

Brief account of this work is: first we give a combinatorial description of exceptional sequences of certain Nakayama algebras by using certain functions. Then by using it, we count the number of exceptional sequences. Our main result is :

\begin{theorem}\label{maintheoremIntro}
There is a one to one correspondence between: idempotent functions from the set of size $n$ to itself and complete exceptional sequences of linear, radical square zero Nakayama algebra of rank $n$.
\end{theorem}

We recall an old result which appears in \cite{idempotent}: the enumeration of idempotent functions is equivalent to the enumeration of forests of rooted labeled trees of height at most one. We use this equivalence to relate all three:

\begin{theorem} The following are equivalent:
\begin{enumerate}[label=\roman*)]
\item the enumeration of complete exceptional sequences of linear radical square zero Nakayama algebras
\item the enumeration of idempotent functions
\item the enumeration of forests of rooted labeled trees of height at most one.
\end{enumerate}
\end{theorem}

We give the answer to the enumeration problem of exceptional sequences by using the equivalences in the previous theorem:
\begin{theorem}\label{enumeration}
The number of complete exceptional sequences of linear, radical square zero Nakayama algebra of rank $n$ is given by the sum $\sum\limits^n_{j=1}\binom{n}{j}j^{n-j}$
\end{theorem}

The organization of the paper is:
in the following section \ref{section2}  we state some  properties of exceptional sequences of Nakayama algebras. In section \ref{section3}, by using results of section \ref{section2}, we will assign an idempotent function to each exceptional sequence and vice versa.
This will be the proof of the main theorem \ref{maintheoremIntro}. In the last section \ref{section4}, we will investigate relationship between exceptional sequences and certain forests. Possible future directions will be given there. The list complete exceptional sequences for $n=4$ which is $41$ is given at the end of the paper.

\section{Exceptional Sequences of Nakayama Algebras}\label{section2}

\subsection{Nakayama Algebras}
Let $\Lambda_n$ be radical square zero Nakayama algebra of rank $n$ with oriented quiver $Q$:\\

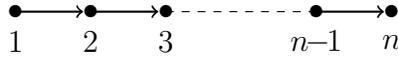
\begin{center}
\begin{tikzpicture}
\draw[thick,->] (0,0) node{$\bullet$} --(0.9,0); 
\draw[thick,->] (1,0) node{$\bullet$} --(1.9,0); 
\draw[dashed,->] (2,0) node{$\bullet$} --(4,0); 
\draw[thick,->] (4,0) node{$\bullet$} --(4.9,0); 
\draw[thick] (5,0) node{$\bullet$} ;
\draw (0,-0.4) node{$1$} ;
\draw (1,-0.4) node{$2$} ;
\draw (2,-0.4) node{$3$} ;
\draw (4,-0.4) node{$n\!\!-\!\!1$} ;
\draw (5,-0.4) node{$n$} ;
\end{tikzpicture}
\end{center}

 We denote the simple $\Lambda_n$ module at vertex $i$ of $Q$ by $[i]$ or $S_i$ and the nonsimple indecomposable projective $\Lambda_n$ module starting at vertex $i$ by $[i,i+1]$. For a given module $X$, $P(X)$, $I(X)$, $\Omega^i(X)$ and $\Sigma^i(X)$ means projective cover, injective envelope, the $i$th syzygy of the projective resolution of $X$ and the $i$th syzygy of the injective resolution of $X$ respectively.
Throughout the text all modules are over $\Lambda_n$ for some $n$ if algebra is not specified.\\
The algebra $\Lambda_n$ has the following Auslander-Reiten quiver:
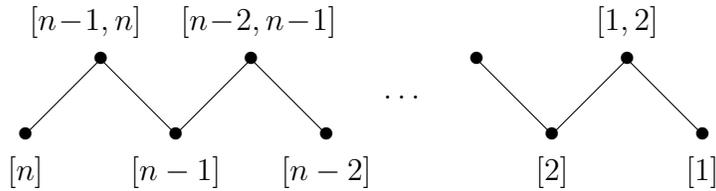
\begin{center}\label{ARquiver}
\begin{tikzpicture}
\draw (0,0) node {$\bullet$}--(1,1) node {$\bullet$}--(2,0) node {$\bullet$}-- (3,1)  node {$\bullet$} --  (4,0) node {$\bullet$} (5,0.5) node {$\ldots$}  (6,1) node {$\bullet$} --(7,0)node {$\bullet$}--(8,1)node {$\bullet$} -- (9,0) node {$\bullet$};
\draw (0,-0.5) node {$[n]$};
\draw (2,-0.5) node {$[n-1]$};
\draw (4,-0.5) node {$[n-2]$};
\draw (7,-0.5) node {$[2]$};
\draw (9,-0.5) node {$[1]$};
\draw (0.8,1.5) node {$[n\!-\!1,n]$};
\draw (3.1,1.5) node {$[n\!-\!2,n\!-\!1]$};
\draw (8,1.5) node {$[1,2]$};
\end{tikzpicture}
\end{center}

Notice that $[1]$ is a simple injective module and $[n]$ is a simple projective module. 

Consider the projective resolution of module $[1]$:

\begin{align}\label{resolution}
\xymatrixcolsep{5pt}
\xymatrix{& [n-1,n]\ar[rr]\ar[rd]     && \ldots   && [2,3]\ar[rr]\ar[rd]   && [1,2]\ar[rr] && [1]\\
[n]\ar[ru] && [n-1] \ar[ru] && \ldots\ar[ru] && [2]\ar[ru]
}
\end{align}
Indeed, the injective resolution of the module $[n]$ gives the same diagram. 
\begin{remark}\label{consequencesofresolution} Resolution \ref{resolution} has immediate consequences:
\begin{itemize}
\item Indecomposable modules are exceptional. 
\item Global dimension of $\Lambda_n$ is $n-1$.
\item For distinct $x$ and $y$, $\Hom_{\Lambda_n}([x,x+1],[y,y+1])\neq 0$ if and only if $x=y+1$.
\item  For given $1\leq x<y\leq n$, $\ext^{i}_{\Lambda_n}([x],[y])\neq 0$ if and only if $i=y-x$.
\item  If $1\leq x<y\leq n$, then $\ext^{i}_{\Lambda_n}([y],[x])= 0$ for all $1\leq i$.
\end{itemize}

\end{remark}

\subsection{Exceptional Sequences}

 A pair of modules $(M,N)$ is called an exceptional pair if $\Hom(N,M)=0$, $\ext^{i}(N,M)=0$, $i\geq 1$ for exceptional $\Lambda_n$ modules $M,N$. A sequence of $\Lambda_n$ modules $(M_1,\ldots,M_t)$ is called an exceptional sequence if for all $1\leq i<j\leq t$, $(M_i,M_j)$ is an exceptional pair. Moreover it is called complete if $t=n$.  Here we want to emphasize that elements of exceptional sequences are isomorphism classes of $\Lambda_n$ modules.

\begin{proposition}\label{orderingofsimples}
Let $E$ be an exceptional sequence, containing two simples $S_x$ and $S_y$. Only one of $(S_x,S_y)$ or $(S_y,S_x)$ can be an exceptional pair.
\end{proposition}

\begin{proof}
By projective resolution \ref{resolution} and remark \ref{consequencesofresolution}, only one of $(S_x,S_y)$ and $(S_y,S_x)$ can be an exceptional pair, because of nontrivial higher extensions.
\end{proof}

It is easy to check that pairs $(P(S),S)$ and $(S,I(S))$ are exceptional where $S$ is a simple module. We generalize this fact below:

\begin{definition} Let $S$ be a simple $\Lambda_n$ module. 
\begin{align}\label{projchain}
(P(\Omega^{t}(S)),P(\Omega^{t-1}(S)),\ldots,P(\Omega(S)),P(S),S)
\end{align} is called \emph{the projective $t$-chain of} $S$ where $0\leq t\leq \pdim S\!\!-\!\!1$ and denoted by $\cP_t(S)$. Similarly, the sequence:
\begin{align}\label{injchain}
 (S,I(S),I(\Sigma(S)),\ldots,I(\Sigma^{t-1}(S)),I(\Sigma^{t}(S)))
\end{align} is called \emph{the injective $t$-chain of} $S$ where $0\leq t\leq \indim S\!\!-\!\!1$ and denoted by $\cI_t(S)$.
\end{definition}

\begin{proposition}\label{chainsareexceptional} Let $S$ be a simple $\Lambda_n$ module. The following sequences are exceptional:
\begin{enumerate}[label=\arabic*)]
\item Projective $t$-chain of $S$ i.e. $\cP_t(S)$

\item Injective $t$-chain of $S$ i.e. $\cI_t(S)$.
\item $(P(\Omega^{t}(S)),\ldots,P(\Omega(S)),P(S),S,I(S),I(\Sigma(S)),\ldots,I(\Sigma^{t'}(S)))$\, for\\ $0\leq t\leq \pdim\!S\!-\!1$ and $0\leq t'\leq \indim\!S\!-\!1$ 
\end{enumerate}

\end{proposition}
\begin{proof}
It is trivial that there is no nontrivial extension between projective modules and simple modules. By AR quiver \ref{ARquiver} and remark \ref{consequencesofresolution}, there is no nontrivial homomorphism in one direction i.e. $\Hom(S,P(S))=0$, $\Hom(I(S),S)=0$, and for $i\geq 0$ $\Hom(P(\Omega^{i}(S)),P(\Omega^{i+1}(S)))=\!0$, $\Hom(I(\Sigma^{i+1}(S)),I(\Sigma^{i}(S)))=0$.
\end{proof}

If there is no need to specify $t$ in \ref{injchain} or \ref{projchain}, we call them simply chain, and drop subscript $t$. Notice that if $t=\pdim S-1$, modules in \ref{projchain} are projective modules of the projective resolution in the given order with reversed maps. We call the third sequence as the projective-injective chain. \\

Let $E$ be a complete exceptional sequence. 
\begin{lemma} E contains at least one simple module.
\end{lemma}
\begin{proof}
Since, a complete exceptional sequence of $\Lambda_n$ has $n$ modules and there are $n-1$ nonsimple indecomposable modules, $E$ has to contain at least one simple module.
\end{proof}

\begin{definition}\label{index} For any module $X$ appearing in an exceptional collection $E$, we define its \emph{index} as its position in $E$. It is denoted by $ind_E(X)$.
\end{definition}

\begin{lemma}\label{MindexS}  Let $M\in E$ be a nonsimple module where $E$ is a complete exceptional sequence of $\Lambda_n$. We have:
\begin{enumerate}[label=\roman*)]
\item If $M$ is an element of a projective chain $\cP(S)$ of a simple module $S$ in $E$, then $ind_E(M)<ind_E(S)$. 
\item If $M$ is an element of an injective chain $\cI(S)$ of a simple module $S$ in $E$, then $ind_E(S)<ind_E(M)$. 
\end{enumerate}

\end{lemma}
\begin{proof}
Consider projective resolution of simple module $S$:
\begin{center}
$P_{\pdim S}\mapsto \ldots\mapsto P_1\mapsto P_0\mapsto S$
\end{center}
Since $M\in\cP(S)$, we can choose $t$ so that $M\cong P_t$. Moreover, all $P_i$ of the resolution $0\leq i\leq t$ are elements of exceptional sequence $E$ by definition of projective chain \ref{projchain}. Existence of map $P_t\mapsto P_{t-1}$ implies $ind_E(P_t)<ind_E(P_{t-1})$, because $(P_{t-1},P_{t})$ is not an exceptional pair. This holds for all consecutive pairs of projective modules: 
\begin{align*}
ind_E(P_{t})&<ind_E(P_{t-1})\\
ind_E(P_{t-1})&<ind_E(P_{t-2})\\
&\mathrel{\makebox[\widthof{=}]{\vdots}}\\
ind_E(P_{1})&<ind_E(P_{0})\\
ind_E(P_{0})&<ind_E(S)
\end{align*}
Hence $ind_E(M)=ind_E(P_t)<ind_E(S)$.\\
By using analogous arguments, the second part of the lemma can be shown. 
\end{proof}

\begin{proposition}\label{everyMisinchain} Let $M\in E$ be a nonsimple module where $E$ is a complete exceptional sequence of $\Lambda_n$.
Then there is exactly one simple module $S$ so that $E$ contains the chain (projective or injective) from $S$ to $M$.

\end{proposition}

\begin{proof} Proof has two parts: existence and then uniqueness of the chain containing $M$.\\
We will show that each nonsimple module appearing in a complete exceptional collection belongs to at least one chain in $E$ by induction on the rank $n$ of $\Lambda_n$. \\
If $n=2$, there are two exceptional sequences containing nonsimple modules: $([2],[1,2])$ and $([1,2],[1])$. It is clear that they form chains.\\

We assume that for complete exceptional sequences of $\Lambda_{n-1}$ the claim holds. Now we analyze the case where rank is $n$. Assume to the contrary that $M$  is not an element of any chain. Therefore simple modules top and socle of $M$ i.e $top(M)$, $soc(M)$ are not elements of $E$. Because pairs $(M,top(M))$ and $(soc(M),M)$ form projective $1$-chain and injective $1$-chain respectively.
\begin{enumerate}[label=\textbf{Case\! \arabic*:}]
\item $M$ is either isomorphic to module $[1,2]$ or $[n-1,n]$. Let $M\cong [n-1,n]$. By assumption, $M$ is not an element of any chain, so, $[n]$ and $[n-1]$ are not in $E$. Then $E'$, the complement of $M$ in $E$, is a complete exceptional sequence of $\Lambda_{n-1}$. Moreover, the module $[n-2,n-1]$ is also not an element of sequence the $E$ since by the induction hypothesis $[n-2,n-1]$ belongs to some chain in complete exceptional sequence $E'$ of $\Lambda_{n-1}$. Therefore, $E'$ is an exceptional sequence for $\Lambda_{n-2}$. But this implies cardinality of $E$:  $\#(E)=1+\#(E')$ is $n-1$. Therefore $E$ cannot be complete. Similar arguments hold for $M\cong [1,2]$.

\item $M$ is nonsimple projective module not isomorphic to $[1,2]$ or $[n-1,n]$. Modules appearing in $E$ except $M$ can be viewed as modules over $\Lambda_a$ and $\Lambda_b$ where $a+b=n-2$, since $top(M)$, $I(top(M))$, $soc(M)$, $P(soc(M))$ are not in $E$ by assumption. Cardinality of $E$: $\#(E)=1+a+b<a+b+2=n$ cannot be $n$, so $E$ is not complete.
\end{enumerate}

Therefore $M$ is an element of at least one chain in $E$.\\Now we prove that $M$ belongs to at most one chain. Assume to the contrary that let $M$ be an element of both $S_x$ and $S_y$ chains with $S_x,S_y\in E$. By proposition \ref{orderingofsimples} we can assume that $x<y$.
There are there possibilities with respect to their positions in $E$:
\begin{enumerate}[label=\textbf{Case\! \arabic*:}]
\item $ind_E(M)<ind_E(S_x)<ind_E(S_y)$.
Notice that $M$ cannot appear in injective chains of both $S_x$, $S_y$ by lemma \ref{MindexS}. Therefore $M\in \cP(S_x)$ and $M\in\cP(S_y)$. This is equivalent to existence of some integers $i,j$ with $i<j$ so that :
\begin{center}
$M\cong P(\Omega^i(S_x))\cong P(\Omega^j(S_y)) \iff$\\
$\Omega^i(S_x)\cong \Omega^j(S_y) \iff$\\
$S_x\cong \Omega^{j-i}(S_y)\iff$\\
$\ext^{j-i}(S_y,S_x)\ncong 0$
\end{center}
Contadiction, since we presumed that $(S_x,S_y)$ is an exceptional pair.
\item $ind_E(S_x)<ind_E(M)<ind_E(S_y)$
By lemma \ref{MindexS}, $M$ is an element of $\cP(S_x)$ and $\cI(S_y)$.  This is equivalent to existence of some integers $i,j$ so that :
\begin{center}
$M\cong I(\Sigma^i(S_x))\cong P(\Omega^j(S_y)) \iff$\\
$P(\Sigma^{i+1}(S_x))\cong \Omega^j(S_y) \iff$\\
$\Sigma^{i+1}(S_x)\cong \Omega^j(S_y)\iff$\\
$\ext^{i+j+1}(S_y,S_x)\ncong 0$
\end{center}
Notice that we used $I(\Sigma^i(S_x))\cong P(\Sigma^{i+1}(S_x))$ which is clear from the resolution \ref{resolution}.
\item $ind_E(S_x)<ind_E(S_y)<ind_E(M)$.
$M$ cannot be an element of projective chains of both $S_x$ and $S_y$ by lemma \ref{MindexS}. Therefore $M\in\cI(S_x)$ and $M\in\cI(S_y)$. This implies the following isomorphisms for some integers $i,j$ with $i>j$:
\begin{center}
$M\cong I(\Sigma^i(S_x))\cong I(\Sigma^j(S_y)) \iff$\\
$\Sigma^i(S_x)\cong \Sigma^j(S_y) \iff$\\
$\Sigma^{i-j}(S_x)\cong S_y\iff$\\
$\ext^{i-j}(S_y,S_x)\ncong 0$
\end{center}
$ind_E(S_x)<ind_E(S_y)$ implies that $x<y$ since otherwise $\ext^{x-y}(S_y,S_x)\neq 0$. But then $E$ contains the injective chain from $S_y$ to $M$ which includes $[x-1,x]$
In $E$, $[x-1,x]$ has to come after $S_y$ and before $S_x$ which is a contradiction.

\end{enumerate}
\end{proof}

\section{Bijection with idempotent functions}\label{section3}
We will construct a map and its inverse between the set of complete exceptional sequences and the set of idempotent functions to show that the map is a bijection.
We recall definition \ref{index} of the index $ind_E(M)$ of a module $M$ in $E$ which  is simply its position.
\subsection{Description of Map }

Let $E=(E_1,E_2,\ldots,E_n)$ be a complete exceptional sequence of $\Lambda_n$.
For each $E_i$, define the map $\Phi_E$ as:
\begin{center}\label{phimap}
$\Phi_E(E_i)=\begin{cases} 
ind_E(E_i)=i &\text{if}\,\, E_i\,\,\text{is a simple module}\\
\multirow{2}{*}{$ind_E(S)$} & \text{if}\,\, E_i\in\cP(S)\,\,\text{or}\,\, E_i\in\cI(S)\,\,\text{where}\,\,S\,\,\text{is simple module in}\,\, E\\ 
    &\text{and}\, E\,\, \text{contains the chain from}\,\, S\,\, \text{to}\,\, E_i\\
\end{cases}$
\end{center}

For any exceptional sequence $E$, we make the assignment $E\mapsto \Phi(E)$ where:\label{assignment}
\begin{center}
$\Phi(E):=\left(\Phi_E(E_1),\Phi_E(E_2),\ldots,\Phi_E(E_n)\right)$
\end{center}

Since $\Phi(E)$ is an array of size $n$, its $i$th component is denoted by $\Phi(E)_i:=\Phi_E(E_i)$. The map $\Phi_E$ is well defined, since each nonsimple module $E_i$ belongs to unique chain by proposition \ref{everyMisinchain}.

\begin{definition}\label{defidempotent} For a given complete exceptional sequence $E$ of $\Lambda_n$, consider the following function from the set of size $n$:
\begin{center}
$f:\left\{1,2,\ldots,n\right\}\longrightarrow\left\{1,2,\ldots,n\right\}$\\\vspace{0.2cm}
$f(i):=\Phi(E)_i$
\end{center}
where $\Phi(E)_i=\Phi_E(E_i)$ defined in \ref{phimap}.
\end{definition}

\begin{proposition} Function $f$ defined in \ref{defidempotent} is an idempotent function on the set $\left\{1,2,\ldots,n\right\}$.
\end{proposition}
\begin{proof}
Since $E$ is exceptional sequence, fixed points of $f$ i.e. $f(x)=x$ are exactly indices of simples in $E$ i.e. $ind_E(S)$. For the remaining points, assume that $f(x)=y$. Because of definition of $f$ in \ref{defidempotent}, $y$ has to be an index of a simple module in $E$ by the map $\Phi_E$. Hence $f(y)=y$. We combine this by composition of $f$:
\begin{align*}
f^2(x)=f(f(x))=f(y)=y
\end{align*}
Since $f(x)=y$, we get $f^2(x)=f(x)$. Therefore for each element of $\left\{1,\ldots,n\right\}$, $f$ is an idempotent function.
\end{proof}

\begin{remark} To denote those functions we can use the assignment given in \ref{assignment}. For example $(2,2,2)$ is the idempotent function with $f(1)=f(2)=f(3)=2$. Therefore for each complete exceptional sequence $E$, we view $n$-tuple $\Phi(E)$ as an idempotent function.
\end{remark}

\begin{example} We give here exceptional sequences of $\Lambda_2$ and $\Lambda_3$ and corresponding idempotent functions due to the map $\Phi$ to be more illustrative. 
\begin{itemize}
\item There are $3$ exceptional sequences of  $\Lambda_2$.
\begin{center}

\begin{tabular}{|c|c|}\hline
$E$ & $\Phi(E)$\\ \hline 
$([1],[2])$ & (1,2)\\ \hline 
$([1,2],[1])$ &(2,2) \\ \hline
$([2],[1,2])$ & (1,1) \\ \hline
\end{tabular}
\end{center}

\item There are 10 exceptional sequences of $\Lambda_3$
\begin{center}

\begin{tabular}{|c|c|}\hline
$E$ & $\Phi(E)$\\ \hline 
$([1],[2],[3])$ & (1,2,3)\\ \hline
$([3],[2,3],[1,2])$ & (1,1,1) \\ \hline
$([2,3],[2],[1,2])$ & (2,2,2) \\ \hline
$([2,3],[1,2],[1])$ &  (3,3,3)\\ \hline
$([1,2],[1],[3])$ & (2,2,3) \\ \hline
$([1],[2,3],[2])$ & (1,3,3) \\ \hline
$([1],[3],[2,3])$ & (1,2,2) \\ \hline
$([2],[1,2],[3])$ & (1,1,3)\\ \hline
$([2],[3],[1,2])$ & (1,2,1)\\ \hline
$([2,3],[1],[2])$ & (3,2,3) \\ \hline
\end{tabular}
\end{center}

\end{itemize}
\end{example}

For example, in the last exceptional sequence $\Phi([2,3])=3$ and not $2$ since the chain $([2,3],[1,2],[1])$ is not contained in $E=([2,3],[1],[2])$.

\subsection{Inverse map}\label{inversemap}

 We will construct a one to one map from idempotent functions to exceptional sequences. Since the two sets are finite, the existence of inverse $\Phi$ map is equivalent to bijection between two sets.

Let $A$ be a $n$-tuple corresponding to an idempotent function. One can show that there exists at least one fixed point of $A$, \cite{idempotent}.
Assume that there are $x$ many distinct ordered fixed points, $p_1,p_2,\ldots,p_x$ i.e. $A(p_i)=p_i$, with $p_1<p_2<\ldots<p_x$. Since $A$ is idempotent, each $p_i$ has to appear in $A$ $c_i$ times, where $1\leq c_i$. Moreover, we can count number of appearances of each $p_i$ comparing its relative position to $p_i$, i.e. $c_i=a_i+b_i+1$ where $a_i$ is number of appearances of nonfixed $p_i$'s \emph{after} fixed $p_i$ and $b_i$ is the number of appearances of nonfixed $p_i$'s \emph{before} fixed $p_i$.

By using array $A$ we will construct a unique exceptional sequence of size $n$. Let $T_1,\ldots,T_x$ be numbers such that $[T_1],\ldots,[T_x]$ are all simple modules of $E$ satisfying $ind_E([T_i])=p_i$ for all $1\leq i\leq x$. We can interpret numbers $a_i,b_i$ in terms of chains i.e.
$a_i$ is the number of nonsimple modules in injective chain $\cI([T_i])$ of $[T_i]$. Similarly, $b_i$ is the number of nonsimple modules in projective chain $\cP([T_i])$ of $[T_i]$. Both chains have transparent structure: by using $a_i$ and $b_i$ we can describe all modules in them.  

Consider the following inequalities for $i$:
\begin{align}\label{ineqI}
T_i-a_i<T_i-a_i+1<\ldots<T_i-1<T_i<T_i+1<\ldots<T_i+b_i
\end{align}
Without loss of generality, here we assume that both $a_i$ and $b_i$ are nonzero. Interpretation of those in terms of the representations of $\Lambda_n$ is:
\begin{align}\label{interpretationofT}
[T_i+j]\cong soc(P(\Omega^{j-1}[T_i]))\\
[T_i-j]\cong top(I(\Sigma^{j-1}[T_i]))\nonumber
\end{align}

Now, we construct inequalities involving distinct $T_i$'s. By proposition \ref{orderingofsimples}, there is a unique choice for the ordering in $E$ i.e. 
\begin{align}\label{ineq1}
T_1<T_2<\ldots<T_x
\end{align}
Since each chain associated to distinct simple $[T_i]$'s are disjoint, we obtain the following inequalities:
\begin{align}\label{ineq2}
T_1+b_1&<T_2-a_2\\
T_2+b_2&<T_3-a_3\nonumber\\
&\mathrel{\makebox[\widthof{=}]{\vdots}}\nonumber\\
T_{x-1}+b_{x-1}&<T_x-a_x\nonumber
\end{align}
These inequalities are consequences of proposition \ref{everyMisinchain} and isomorphisms \ref{interpretationofT}. By simple observation, we obtain two more inequalities:
\begin{align}\label{ineq3}
0&<T_1-a_1\\
T_x+b_x&<n+1\nonumber
\end{align}

To find indices $T_1,\ldots,T_x$ we need to find a common solution to the system of inequalities given in \ref{ineqI},\ref{ineq1},\ref{ineq2} and \ref{ineq3}. This is relatively easy, because there are $n$ terms in total, the unique solution is $1<2<\ldots<n$. Therefore, we get:
\begin{align*}
T_1&=1+a_1\\
T_2&=T_1+a_2+b_1+1\\
T_3&=T_2+a_3+b_2+1\\
&\mathrel{\makebox[\widthof{=}]{\vdots}}\\
T_x&=T_{x-1}+a_{x}+b_{x-1}+1\\
\end{align*}

After placing each simple module $[T_i]$ to position $p_i$, one can place nonsimple modules according to positions of $a_i$'s and $b_i$'s. We did not specify them to avoid complications, however the structure of those modules are clear from the isomorphisms \ref{interpretationofT}. Moreover the existence of a unique solution implies that the exceptional sequence $E$ is unique, hence the assignment $\Gamma:A\mapsto E$ is injection for distinct idempotent functions.

\begin{example} 
Let $A=(7,2,2,4,4,7,7)$ be an idempotent function. There are $3$ fixed points: $p_1=2$, $p_2=4$ and $p_3=7$. Other ingredients are $a_1=1,b_1=0,a_2=1,b_2=0,a_3=0,b_3=2$. We want to find simple modules $[T_1],[T_2],[T_3]$ satisfying:
\begin{align}
0<T_1-1<T_1<T_2-1<T_2<T_3<T_3+1<T_3+2<8
\end{align}
The unique solution is: $T_1=2$, $T_2=4$ and $T_3=5$. Therefore at first we place simple modules into the exceptional sequence $E$:
\begin{align*}
(-,[2],-,[4],-,-,[5])
\end{align*}
By using $A(3)=2$, $A(5)=5$ and $A(1)=A(6)=7$ we get: $E_3=[1,2]$, $E_5=[3,4]$, $E_1=[6,7]$ and $E_6=[5,6]$. Hence the exceptional sequence $E$ obtained by function $A$ is :
\begin{center}
$([6,7],[2],[1,2],[4],[3,4],[5,6],[5])$
\end{center}
\end{example}

\begin{proposition} For each idempotent function $A$, let $\Gamma$ be map: 
\begin{align}
\Gamma:A\longrightarrow E
\end{align}
obtained by construction in subsection \ref{inversemap}. Then $\Phi$ and $\Gamma$ are inverse functions of each other.
\end{proposition}
\begin{proof}
First, we want to show that $\Gamma\circ\Phi$ is the identity. Let $E$ be complete exceptional sequence of $\Lambda_n$. If we apply $\Gamma$ to idempotent function $\Phi(E)$, we observed that solution to system of inequalities \ref{ineq1},\ref{ineq2}, \ref{ineq3} is unique. Exceptional sequence $E$ satisfies those, therefore $\Gamma\circ\Phi$ is identity on the set of complete exceptional sequences of size $n$. This implies $\Gamma$ is surjective. But we already know it is injective. So, it is a bijection and so is $\Phi=\Gamma^{-1}$.
\end{proof}

Now we restate and prove the main result \ref{maintheoremIntro} of the paper:

\begin{theorem}\label{maintheorem}
There is a bijection between the set of complete exceptional sequences of linear radical square zero Nakayama algebras $\Lambda_n$ and the set of idempotent functions on the set $\left\{1,2,\ldots,n\right\}$ which is given by the map $\Phi$ defined in \ref{phimap}.
\end{theorem}
\begin{proof}
In the previous proposition we constructed an inverse to the map $\Phi$. Since two sets are finite, $\Phi$ and $\Gamma$ are bijections.
\end{proof}

\begin{theorem}
The number of exceptional sequences over linear radical square zero Nakayama algebra $\Lambda_n$ is given by $\sum_{j\geq 1} \binom{n}{j}j^{n-j}$.
\end{theorem} 

\begin{proof}
By the theorem \ref{maintheorem}, it is equivalent to counting the number idempotent functions. Here we give a sketch of proof which gives number of idempotent functions. Detailed version can be found at \cite{idempotent}. Let $f$ be an idempotent function. The fixed point set is nonempty. Assume that there are $j$ fixed points for $f$ and $n-j$ nonfixed points. The total number of ways to choose fixed points is $\binom{n}{j}$. The number of maps which does not have any fixed point is given by $j^{n-j}$. Therefore we get $\sum_{j\geq 1} \binom{n}{j}j^{n-j}$.
\end{proof}

\section{Height at most one Forests}\label{section4}
In the work of Riordan \cite{forests} table 1 includes the number of height at most one forests with $n$ roots. Its relation with idempotent functions explicity stated in \cite{idempotent}. We rely on those to state extended result:
\begin{theorem}
The following three sets are equivalent to each other:
\begin{itemize}
\item The set of labeled height at most one forests with $n$ labeled vertices
\item The set of idempotent endofunctions on the set $\left\{1,2,\ldots,n\right\}$ 
\item The set of complete exceptional sequences of linear radical square zero Nakayama algebra $\Lambda_n$.
\end{itemize}
\end{theorem}

One can construct a direct bijection between the first and the last items, however we want to emphasis another fact by using forests: it is relatively simpler to encode structure of the exceptional sequences by using structure of forests in the following way:
\begin{enumerate}[label=\roman*)]
\item Roots correspond to simple modules in the exceptional sequence.
\item The number of leaves coming out each distinct root corresponds to chain involving that root. 
\end{enumerate}

\begin{example} If $n=3$, there are $3$ unlabeled forests:
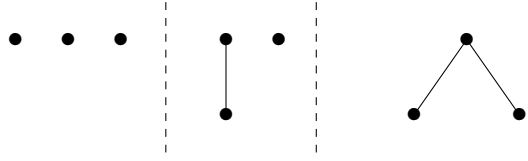
\begin{center}
\begin{tikzpicture}
\draw (0,1) node{$\bullet$} (0.7,1) node{$\bullet$} (1.4,1) node{$\bullet$} ;
\draw[dashed] (2,1.5)--(2,-0.5);
\draw (2.8,1) node{$\bullet$} --(2.8,0) node{$\bullet$} (3.5,1) node{$\bullet$};
\draw[dashed] (4,1.5)--(4,-0.5);
\draw (5.3,0) node {$\bullet$}--(6,1) node{$\bullet$}--(6.7,0) node{$\bullet$};
\end{tikzpicture}
\end{center}
When $n=4$, there are $5$ unlabeled forests of height at most one. 
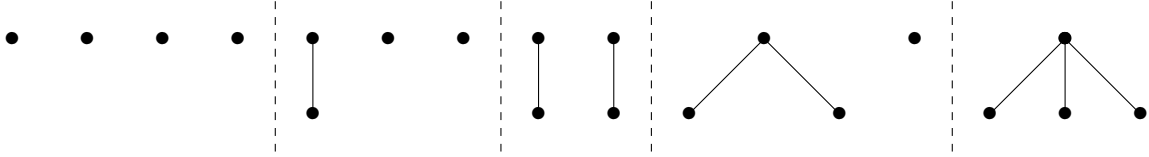
\begin{center}
\begin{tikzpicture}
\draw[dashed] (3.5,1.5)  -- (3.5,-0.5);
\draw[dashed] (6.5,1.5) -- (6.5,-0.5) ;
\draw[dashed] (8.5,1.5) -- (8.5,-0.5);
\draw[dashed] (12.5,1.5)  -- (12.5,-0.5) ;

\draw (0,1) node {$\bullet$};
\draw (1,1) node {$\bullet$};
\draw (2,1) node {$\bullet$};
\draw (3,1) node {$\bullet$};
\draw (4,1) node {$\bullet$}--(4,0) node {$\bullet$};
\draw (5,1) node {$\bullet$};
\draw (6,1) node {$\bullet$};

\draw (7,1) node {$\bullet$}--(7,0) node {$\bullet$};
\draw (8,1) node {$\bullet$}--(8,0) node {$\bullet$};
\draw (9,0) node {$\bullet$}--(10,1) node {$\bullet$}--(11,0) node {$\bullet$};
\draw (12,1) node {$\bullet$};
\draw (13,0) node {$\bullet$}--(14,1) node {$\bullet$};
\draw (14,0) node {$\bullet$}--(14,1) node {$\bullet$};
\draw (15,0) node {$\bullet$}--(14,1) node {$\bullet$};
\end{tikzpicture}
\end{center}
\end{example}

For instance, both exceptional sequences $([2],[1,2],[4],[3,4])$ and $([1],[4],[3,4],[2,3])$ have two simple modules. However, their forest types are 3rd and 4th respectively because length of the chains are different.

\subsection{Future Directions}
We recall the result of Seidel \cite{seidel}, which states that the number of exceptional sequences of $\mathbb{A}_n$ type is given by $(n+1)^{n-1}$. Indeed it is the number of all rooted labeled forests with $n$ labels, known as Cayley's formula \cite{stanley}. By that result and results of this paper we consider the following:
\begin{question} Is it possible to interpret complete exceptional sequences of connected linear Nakayama algebras of rank $n$ by certain rooted labeled forests with $n$ labels?
\end{question}
\bibliographystyle{alpha}

\begin{thebibliography}{RBG{\etalchar{+}}90}

\bibitem[Ara13]{araya}
Tokuji Araya.
\newblock Exceptional sequences over path algebras of type $\mathbb{A}_n$ and non-crossing
  spanning trees.
\newblock {\em Algebras and Representation Theory}, 16(1):239--250, 2013.

\bibitem[CB93]{craw}
William Crawley-Boevey.
\newblock Exceptional sequences of representations of quivers.
\newblock {\em Representations of algebras (Ottawa, ON, 1992)}, 14:117--124,
  1993.

\bibitem[GIMO15]{igusa}
Alexander Garver, Kiyoshi Igusa, Jacob Matherne, and Jonah Ostroff.
\newblock Combinatorics of exceptional sequences in type $\mathbb{A}$.
\newblock {\em arXiv preprint arXiv:1506.08927}, 2015.

\bibitem[GR{\etalchar{+}}87]{goro}
Alexei Gorodentsev, Alexei Rudakov.
\newblock Exceptional vector bundles on projective spaces.
\newblock {\em Duke Mathematical Journal}, 54(1):115--130, 1987.

\bibitem[HP17]{hille}
Lutz Hille and David Ploog.
\newblock Exceptional sequences and spherical modules for the Auslander algebra
  of $k[x]/x^t$.
\newblock {\em arXiv preprint arXiv:1709.03618}, 2017.

\bibitem[HS67]{idempotent}
Bernhard Harris and Lowell Schoenfeld.
\newblock The number of idempotent elements in symmetric semigroups.
\newblock {\em Journal of Combinatorial Theory}, 3(2):122--135, 1967.

\bibitem[IS10]{ralf}
Kiyoshi Igusa and Ralf Schiffler.
\newblock Exceptional sequences and clusters.
\newblock {\em Journal of Algebra}, 323(8):2183--2202, 2010.

\bibitem[LM09]{len}
Helmut Lenzing and Hagen Meltzer.
\newblock Exceptional pairs in hereditary categories.
\newblock {\em Communications in Algebra}, 37(8):2547--2556, 2009.

\bibitem[Mel95]{melt}
Hagen Meltzer.
\newblock Exceptional sequences for canonical algebras.
\newblock {\em Archiv der Mathematik}, 64(4):304--312, 1995.

\bibitem[ONS{\etalchar{+}}13]{ringel}
Mustafa Obaid, Khalid Nauman, Wafa Al Shammakh, Wafaa Fakieh, and
  Claus~Michael Ringel.
\newblock The number of complete exceptional sequences for a Dynkin algebra.
\newblock {\em arXiv preprint arXiv:1307.7573}, 2013.


\bibitem[Rin94]{ringel2}
Claus Michael Ringel.
\newblock The braid group action on the set of exceptional sequences of a hereditary Artin algebra.
\newblock {\em Contemporary Mathematics}, v171, 1994.




\bibitem[RBG{\etalchar{+}}90]{rudakov}
Alexei~N Rudakov, AI~Bondal, AL~Gorodentsev, BV~Karpov, MM~Kapranov,
  SA~Kuleshov, AV~Kvichansky, D~Yu Nogin, and SK~Zube.
\newblock {\em Helices and vector bundles: Seminaire Rudakov}, volume 148.
\newblock Cambridge University Press, 1990.

\bibitem[Rio68]{forests}
John Riordan.
\newblock Forests of labeled trees.
\newblock {\em Journal of Combinatorial Theory}, 5(1):90--103, 1968.

\bibitem[Sei01]{seidel}
Uwe Seidel.
\newblock Exceptional sequences for quivers of Dynkin type.
\newblock {\em Communications in Algebra}, 29(3):1373--1386, 2001.

\bibitem[Sta99]{stanley}
Richard~P Stanley.
\newblock Enumerative combinatorics, vol. 2. 1999.
\newblock {\em Cambridge Stud. Adv. Math}, 1999.

\end{thebibliography}
\newcommand{\etalchar}[1]{$^{#1}$}

\newpage
\voffset -4cm
\textheight 22.5truecm

\begin{align*}
\underline{Exceptional\quad\!\! Seq.\quad\!\! of\quad\!\! \Lambda_4} &\hspace{2cm} \underline{\Phi(E)}\\
[[1],[2,3],[2],[4]] &\hspace{2cm}(1,3,3,4)\\
[[3,4],[3],[2,3],[1,2]] &\hspace{2cm}(2,2,2,2)  \\
[[3,4],[2],[3],[1,2]]  &\hspace{2cm}(3,2,3,2) \\
[[3,4],[2],[1,2],[3]]&\hspace{2cm}(4,2,2,4)  \\
[[3,4],[1],[3],[2,3]]&\hspace{2cm}(3,2,3,3)  \\
[[3,4],[1],[2],[3]]&\hspace{2cm}(4,2,3,4)  \\
[[3,4],[1],[2,3],[2]]&\hspace{2cm}(4,2,4,4)  \\
[[3,4],[2,3],[2],[1,2]]&\hspace{2cm}(3,3,3,3)  \\
[[3,4],[2,3],[1],[2]]&\hspace{2cm}(4,4,3,4)  \\
[[2],[4],[1,2],[3,4]]&\hspace{2cm}(1,2,1,2)  \\
[[2],[4],[3,4],[1,2]] &\hspace{2cm}(1,2,2,1) \\
[[3],[2,3],[1,2],[4]]&\hspace{2cm}(1,1,1,4)  \\
[[3],[2,3],[4],[1,2]] &\hspace{2cm}(1,1,3,1) \\
[[3],[4],[2,3],[1,2]]  &\hspace{2cm}(1,2,1,1)\\
[[4],[3,4],[2,3],[1,2]] &\hspace{2cm}(1,1,1,1) \\
[[1],[3,4],[2,3],[2]] &\hspace{2cm}(1,4,4,4) \\
[[1],[3,4],[2],[3]]  &\hspace{2cm}(1,4,3,4)\\
[[1],[3,4],[3],[2,3]] &\hspace{2cm}(1,3,3,3) \\
[[1],[2],[3,4],[3]] &\hspace{2cm}(1,2,4,4) \\  
[[1],[2],[4],[3,4]] &\hspace{2cm}(1,2,3,3) \\
[[1],[3],[2,3],[4]] &\hspace{2cm}(1,2,2,4) \\
[[1],[3],[4],[2,3]]  &\hspace{2cm}(1,2,3,2)\\
[[1],[4],[3,4],[2,3]] &\hspace{2cm}(1,2,2,2) \\
[[2],[1,2],[3,4],[3]] &\hspace{2cm}(1,1,4,4) \\
[[2],[1,2],[3],[4]] &\hspace{2cm}(1,1,3,4) \\
[[2],[3,4],[1,2],[3]] &\hspace{2cm}(1,4,1,4) \\
[[2],[3,4],[3],[1,2]] &\hspace{2cm}(1,3,3,1) \\
[[2],[3],[1,2],[4]] &\hspace{2cm}(1,2,1,4) \\
[[2],[3],[4],[1,2]]  &\hspace{2cm}(1,2,3,1)\\
[[3,4],[2,3],[1,2],[1]] &\hspace{2cm}(4,4,4,4) \\
[[3,4],[1,2],[1],[3]] &\hspace{2cm}(4,3,3,4) \\
[[2,3],[2],[4],[1,2]] &\hspace{2cm}(2,2,3,2) \\
[[2,3],[2],[1,2],[4]] &\hspace{2cm}(2,2,2,4) \\
[[2,3],[1],[2],[4]] &\hspace{2cm}(3,2,3,4) \\
[[2,3],[1,2],[1],[4]] &\hspace{2cm}(3,3,3,4) \\
[[1,2],[1],[4],[3,4]]  &\hspace{2cm}(2,2,3,3)\\
[[1,2],[1],[3],[4]] &\hspace{2cm}(2,2,3,4) \\
[[1,2],[1],[3,4],[3]] &\hspace{2cm}(2,2,4,4) \\
[[1,2],[3,4],[1],[3]] &\hspace{2cm}(3,4,3,4) \\
[[1],[2],[3],[4]] &\hspace{2cm}(1,2,3,4) \\
[[2],[1,2],[4],[3,4]] &\hspace{2cm}(1,1,3,3) \\
\end{align*}

\end{document}